\documentclass{amsart}

\usepackage{graphicx}
\usepackage{amsmath, amsthm, amssymb, amsfonts}
\usepackage{float}
\usepackage{qtree}

\newtheorem{theorem}{Theorem}[section]
\newtheorem*{theorem*}{Theorem}
\newtheorem{lemma}[theorem]{Lemma}
\newtheorem{corollary}[theorem]{Corollary}

\theoremstyle{definition}

\theoremstyle{remark}

\numberwithin{equation}{section}

\newcommand{\Z}{{\mathbb Z}}

\begin{document}

\title[The reciprocal sum of pnds]{The reciprocal sum of primitive nondeficient numbers}
%Improved bounds on random probable primes]{Improved bounds on the probability that a random probable prime is composite}

\author{Jared Duker Lichtman}
\address{Department of Mathematics, Dartmouth College, Hanover, NH 03755}

\email{jdl.18@dartmouth.edu}
%\thanks{This paper is based on a portion of the author's undergraduate thesis \cite{JDLt}, conducted under the guidance of Professor Carl Pomerance at Dartmouth College.}

%\author{Carl Pomerance}
%\address{Department of Mathematics, Dartmouth College, Hanover, NH 03755}
%\email{carl.pomerance@dartmouth.edu}

%    General info

%\date{\today}
\date{February 12, 2018.}
%\date{January 1, 2001 and, in revised form, June 22, 2001.}

%\dedicatory{This paper is dedicated to}


\begin{abstract}
We investigate the reciprocal sum of {\it primitive nondeficient numbers}, or pnds. In 1934, Erd\H{o}s showed that the reciprocal sum of pnds converges, which he used to prove that the abundant numbers have a natural density. We show the reciprocal sum of pnds is between 0.348 and 0.380.
\end{abstract}

\maketitle

%% The correct journal style for \specialsection is all uppercase; a known bug
%% in amsart.cls prevents this, so input must be uppercase until it is fixed.

\section{Introduction}

The field of probabilisitic number theory got its start in the 1920s and 1930s with work of Schoenberg \cite{schoenb}, Davenport \cite{davenp}, and others who proved the existence of distribution functions for $\varphi(n)/n$, $\sigma(n)/n$, and similar functions.  (Here, $\varphi$ is Euler's function and $\sigma$ is the sum-of-divisors function.)  This line of work led up to the celebrated theorems of Erd\H os--Wintner and Erd\H os--Kac. The present paper is concerned with one of the earlier results in this history, namely Erd\H os \cite{erdos1}, where an elementary argument is
presented to show that the density of the set of $n$ with $\sigma(n)/n\ge 2$ exists.

With nomenclature going back to ancient times, a number $n$ is said to be abundant, deficient, or perfect if $\sigma(n)/n$ is greater than, less than, or equal to 2, respectively. One then says $n$ is nondeficient if $\sigma(n)/n\ge2$. For such $n$, since $\sigma(n)/n = \sum_{d\mid n}1/d$, all multiples of $n$ are also nondeficient. This naturally leads one to consider nondeficient numbers all of whose proper divisors are deficient, so-called {\it primitive nondeficient} (pnd) numbers. The sequence of pnds is OEIS A281505.

It is easy to see that every nondeficient number has a pnd divisor. The proof of Erd\H os \cite{erdos1} then hinged on showing that the reciprocal sum of the pnds is convergent. The convergence was shown by determining a sufficiently small upper bound on the counting function for pnds.  Denoting the number of pnds $\leq x$ by $N(x)$, the paper showed that
\[N(x)=o\left(\frac{x}{(\log x)^2}\right),\]
which is enough to prove that the sum of reciprocals of the pnds converges.  A more detailed study by Erd\H os in \cite{erdos2} found that, for sufficiently large $x$,
\[x\exp(-c_1\sqrt{\log x\log\log x})\leq N(x)\leq x\exp(-c_2\sqrt{\log x\log\log x}), \]
where $c_1=8$ and $c_2=1/25$.  Presumably there is a constant $c$ such that $c_1,c_2=c+o(1)$
as $x\to\infty$.  Recent numerical experiments
of Silva \cite{silva} suggest such a $c$ may be close to 1.
The best that is now known asymptotically is a result of Avidon \cite{avidon} who showed we may take $c_1=\sqrt 2+\epsilon$ and $c_2=1-\epsilon$ for any fixed $\epsilon>0$.

Once a series is found to converge, it is natural to wonder what its value may be. For example, by Brun's Theorem it is known that the reciprocal sum of twin primes converges.  This sum, called Brun's constant, is approximately $1.902160583104$, which is found by extrapolating
via the Hardy--Littlewood heuristics.  However, the best proven upper bound is $2.347$, see \cite{cp, klyve}.  Similarly, Pomerance \cite{pomerance} proved that the reciprocal sum of numbers in amicable pairs converges, and work has also been done to determine bounds on this value, the Pomerance constant, the current bounds being $0.0119841556$ and $222$, see \cite{bk, NP}. Given the existing nomenclature, we call the value of the reciprocal sum of pnds the \emph{Erd\H os constant}. 

The principal result of this paper is the following theorem.

\begin{theorem} \label{thm:final}
The Erd\H{o}s constant $\sum_{n\textnormal{ is a pnd}}1/n$ lies in the interval
\begin{align*}
(0.34842 \ , \ 0.37937).
\end{align*}
\end{theorem}

The lower bound was obtained by direct calculation of the partial sum over pnds up to $10^{14}$, carried out by Silva \cite{silva}. For the upper bound, our starting point is an estimate of Erd\H os \cite{erdos2} on the distribution of pnds. We make such estimates explicit, and along the way we sharpen the original argument by leveraging properties of primitive sets, that is, sets of numbers in which none divides any other. 

We note that a simple-minded application of pnd distribution estimates yields an upper bound greater than 2000. With strategic choices of a smoothness cutoff one may reduce the bound to just over 300. In the author's thesis \cite{JDLt}, the upper bound was reduced to under 19 by drawing on {\it upper bound} methods in \cite{deleglise} for the density of abundant numbers, often denoted by $\Delta$. The new, and perhaps surprising, idea to this paper is to replace the latter argument by those inspired by {\it lower bound} methods for $\Delta$. From \cite{mits2}, $\Delta$ may be expressed as a convergent series over the pnds, which turns out to have a natural relation with the desired pnd reciprocal sum in an interval. This approach leverages knowledge of $\Delta$ to high precision. Specifically, the abundant density is tightly bounded by
$$0.2476171 < \Delta < 0.2476475,$$
as shown in \cite{mits1}.
Incorporating these new ideas cuts the final upper bound down to under $0.38$.

For $\alpha \ge 1$, the set of numbers $n$ with $\sigma(n)/n > \alpha$, the so-called $\alpha$-abundants, has a positive density $\Delta_\alpha$. The $\alpha$-perfects and $\alpha$-deficients are defined similarly. So, an $\alpha$-pnd is an $\alpha$-nondeficient number all of whose divisors are $\alpha$-deficient. The abundant density method in \cite{mits1} was shown to generalize to $\alpha$-abundants. For the pnd distribution estimates, by Lemma 2.1 in \cite{mitspaul} there is a positive constant $c_\alpha$ so that, for sufficiently large $x$, the number of $\alpha$-pnds up to $x$ is at most
$$x/\exp(c_\alpha \sqrt{\log x \log_2 x})$$
if $\alpha$ is not a Liouville number. In particular, by partial summation this implies that the reciprocal sum of $\alpha$-pnds converges for any non-Liouville $\alpha$. On the other hand, Erd\H os \cite{erdos3} constructed a family of Liouville numbers $\alpha$ whose $\alpha$-pnd reciprocal sum diverges. In line with these results, the methods in this paper naturally generalize to the reciprocal sum of $\alpha$-pnds in the non-Liouville case.

\subsection*{Notation}
We use $p$ and $q$ to denote prime numbers. We write $(n,m)$ to denote the greatest common divisor of $n$ and $m$. Let $h(n) = \sigma(n)/n$. For positive integers $n$, we let $P(n), \omega(n)$ denote the largest prime factor and the number of distinct prime factors of $n$, respectively. By convention $P(1)=1$. We say a positive integer $n$ is $y$-smooth if $P(n)\le y$. We say $n$ is square-full if $p^2\mid n$ for all primes $p\mid n$. Let $\gamma$ denote the Euler-Mascheroni constant. Let $\log_k n=\log(\log_{k-1} n)$ denote the $k$-fold logarithm. Let $\textrm{Li}(x) = \int_2^x\frac{dt}{\log t}$. In many instances, we take a sum over certain subsets of pnds, in
which cases we use $\sum^{'}_{n}$ to denote $\sum_{n\textrm{ is a pnd}}$.

\section{Setting up the bound}
To bound the reciprocal sum of pnds, we may first compute the sum directly up to some convienient $x_0\in\Z$, so that
\begin{align}
\varepsilon = \sideset{}{'}\sum_{n}\frac{1}{n} & = \sideset{}{'}\sum_{n\le x_0}\frac{1}{n} + \sideset{}{'}\sum_{n> x_0}\frac{1}{n}.
\end{align}
Silva \cite{silva} computed that for $x_0=10^{14}$,
\begin{align}\label{Silva}
\sideset{}{'}\sum_{n\le 10^{14}}\frac{1}{n} = 0.34842\ldots,
\end{align}
as well as that the number of pnds up to $10^{14}$ is 870510225.\footnote{Silva \cite{silva} provided the first 80 digits of \eqref{Silva} to be\\ 0.3484218159391501691221675470639682348139\ldots}
Before moving on, we note that we expect this lower bound to well-approximate the entire sum. Indeed, if one approximates the tail by $\exp(-\sqrt{\log x\log\log x})$, roughly the asymptotic upper bound, partial summation gives
\begin{align*}
\sideset{}{'}\sum_{n> 10^{12}}\frac{1}{n} \approx \int_{10^{14}}^\infty\frac{dx}{x}\exp(-\sqrt{\log x\log\log x}) = \int_{14\log10}^\infty e^{-\sqrt{t\log t}}\;dt < 1.3\cdot10^{-4},
\end{align*}
which suggests that the true value of the Erd\H{o}s constant is approximately $0.3485$.

For the remaining part of the series, we shall split up by $y$-smoothness, for $y$ to be determined. For the $y$-smooth contribution, we will adapt an elementary method pioneered by Rankin \cite{rankin}. See \cite{JDLCP2} for a more detailed study of explicit estimates for smooth numbers. For the non-$y$-smooth contribution, we first study the related sum
\begin{align*}
M(x,y) & = \sideset{}{'}\sum_{\substack{n\le x\\ P(n)> y}}1.
\end{align*}

\subsection{An upper bound for $M(x,y)$}
Every integer $n$ may be decomposed uniquely into $n=qs$, where $s=s(n)$ is square-full, $q=q(n)$ is square-free, and $(s,q)=1$. We first prove a preliminary lemma for pnds with large square-full part. 
\begin{lemma} \label{lmsqfull}
Let $\lambda=\zeta(3/2)/2\zeta(3)$. For $x>y>8$, we have
\begin{align*}
B(x,y):=\sideset{}{'}\sum_{\substack{n\le x\\s(n)\ge y}}1 \le \lambda xy^{-1/2}+2xy^{-2/3}.
\end{align*}
\end{lemma}
\begin{proof}
Denoting $K(y)$ as the number of square-full integers up to $y$, by (8) in \cite{golomb} we have
\begin{align}
-3\sqrt[3]{y} \le K(y) - 2\lambda\sqrt{y} \le 0,\qquad \lambda=\frac{\zeta(3/2)}{2\zeta(3)}.
\end{align}
For each square-full number $s\in [y, x]$, the set $\{m\in[1,x/s]: ms\textrm{ is a pnd}\}$ is primitive so it contains at most $x/2s+1$ elements. Then by partial summation,
\begin{align*}
    \sideset{}{'}\sum_{\substack{n\le x\\s(n)\ge y}}1 & \le \sum_{\substack{y<s\le x\\s\;\square\text{-full}}}\Big(\frac{x}{2s}+1\Big) \le K(x) + \frac{x}{2}\sum_{\substack{s>y\\s\;\square\text{-full}}}\frac{1}{s}\\
    & \le \frac{x}{2}\Big(-\frac{K(y)}{y} + \int_{y}^\infty \frac{K(t)}{t^2}\;dt \Big) + 2\lambda x^{1/2}\\
    & \le \frac{x}{2}\Big(-\frac{2\lambda y^{1/2}-3y^{1/3}}{y} + 2\lambda\int_{y}^\infty \frac{dt}{t^{3/2}}\Big) + 2\lambda x^{1/2}\\
    &= \frac{x}{2}\Big(-2\lambda y^{-1/2}+3xy^{-2/3} + 2\lambda\Big[-2t^{-1/2}\Big]_{y}^\infty\Big) + 2\lambda x^{1/2}\\
    & = \lambda xy^{-1/2}+3/2xy^{-2/3}+2\lambda x^{1/2}  \le \lambda xy^{-1/2}+2xy^{-2/3}.
\end{align*}
which completes the proof.
\end{proof}

Now to bound $M(x,y)$, we roughly follow the developments in \cite{erdos2}, and split up into two cases. In the first case, suppose $s(n)\ge y^a$ for some parameter $a\in (0,\tfrac{1}{3})$ to be determined later. Then by Lemma \ref{lmsqfull}, we have
\begin{align} \label{eq:sqfull}
\sideset{}{'}\sum_{\substack{n\le x\\ P(n) > y\\s(n)\ge y^a}}1 \le \sum_{\substack{n\le x\\ s(n)\ge y^a}}1 = B(x,y^a) \le \lambda xy^{-a/2}+2xy^{-2a/3}.
\end{align}

In the second case, we have $s(n)< y^{a}$. We prove a lemma, adapted from \cite{erdos2}.

\begin{lemma}\label{lmErdos}
Assume $x>y>8$. Let $\{n_1,\ldots,n_m\}$ be the set of pnds $n_i\le x$ with $P(n_i)>y$ and square-full part $s(n_i)< y^a$. Let $b=3a/2$ and suppose that $b$ satisfies
\begin{align} \label{eq:dagger}
b\in\big(0,\tfrac{1}{2}\big),\qquad 2 \ge \Big(2-y^{-b}\Big)\Big(1+\sqrt{2/y}\Big)^{2\log x/\log(y/2)}.
\end{align}
Then for each $n_i$ there exists a square-free divisor $d_i\mid n_i$ with $d_i\in [y^{b/3}, \frac1{\sqrt2}y^{1/2}]$. Moreover, $\{n_1/d_1,\ldots,n_m/d_m\}$ is a set of $m$ distinct integers at most $x y^{-b/3}$, and therefore 
\begin{align}
\sideset{}{'}\sum_{\substack{n\le x\\P(n)>y\\s(n) < y^a}}1 = m\le x y^{-b/3}.
\end{align}
\end{lemma}

\begin{proof}
To prove existence, we proceed by contradiction. Take any $n=n_i$, and suppose there is no divisor $d\mid q(n)$ with $d\in [y^{b/3} , \frac{1}{\sqrt{2}}y^{1/2}]$. Then since $s(n)<y^a=y^{2b/3}$, there are no prime divisors $p\mid n$ in the interval $[y^{b/3} , \frac{1}{\sqrt{2}}y^{1/2}]$. We may therefore decompose $n$ as $n=uv$ where $q< y^{b/3}$ and $p>\frac{1}{\sqrt{2}}y^{1/2}$ for all primes $q\mid u$ and $p\mid v$, respectively. We have that
\begin{align*}
    \omega(v) = \sum_{p\mid v}1 \le \sum_{p\mid v}\frac{\log p}{\log(y/2)/2} & = \frac{2}{\log(y/2)}\log \prod_{p\mid v}p\\
    & = \frac{2\log v}{\log(y/2)} \le \frac{2\log x}{\log(y/2)}.
\end{align*}

Suppose $u\le y^b$, where we recall that $b$ satisfies \eqref{eq:dagger}. Also recall $y<P(n)\le n$ so $u<n$ is a proper divisor of the pnd $n$. Thus $u$ is deficient so $\sigma(u)\le 2u-1$, and since the function $h(n)=\sigma(n)/n$ is multiplicative,
\begin{align*}
    2 \le h(n) = h(u)h(v) &\le \Big(2-\frac{1}{u}\Big)\prod_{p\mid v} \Big(1+\frac{1}{p}\Big)\\
    &\le \Big(2-y^{-b}\Big)\Big(1+\sqrt{2}y^{-1/2}\Big)^{\omega(v)}\\
    &< \Big(2-y^{-b}\Big)\Big(1+\sqrt{2}y^{-1/2}\Big)^{2\log x/\log(y/2)}.
\end{align*}
This contradicts the assumption \eqref{eq:dagger} for $b$. Hence we deduce $u> y^{b}$.

Write the square-free part of $u$ as $q(u)=q_1q_2\cdots q_t$ in ascending order of primes. Since $s(n)< y^{a} < \frac{1}{\sqrt{2}}y^{1/2}$, all the primes of $s(n)$ are less than $\frac{1}{\sqrt{2}}y^{1/2}$. Therefore $s(n)\mid u$ and so $s(n)=s(u)$. Then since $u>y^{b}$ by the preceding paragraph,
$$q(u)=\frac{u}{s(u)} = \frac{u}{s(n)} > y^{b-a}=y^{b/3}.$$
Since $q_i<y^{b/3}$ for each prime $q_i\mid u$, there must exist some index $l\in[1,t]$ such that 
$$q_1\cdots q_{l-1}\le y^{b/3} < q_1\cdots q_l < y^{2b/3}.$$
Since $b<1/2$ and $y>8$, we have $y^{2b/3} < \tfrac{1}{\sqrt 2}y^{1/2}$. Thus $d=q_1\cdots q_l$ is a square-free divisor of $n$ in the interval $[ y^{b/3}, \tfrac{1}{\sqrt 2}y^{1/2}]$. However, this contradicts our assumption. Hence each $n_i$ has a square-free divisor $d_i$ in the interval.

The proof of distinctness is unchanged from \cite{erdos2}, but we provide it for completeness. For all $n=n_i$, since the square-full part $s(n)$ is less than $y^{a}<y<P(n)$ we have that $P(n)^2$ does not divide $n$ so
\begin{align*}
    2 \le h(n) & = h(P(n))h\Big(\frac{n}{P(n)}\Big)\\
    & = \Big(1+\frac{1}{P(n)}\Big)h\Big(\frac{n}{P(n)}\Big) \le 2 + \frac{2}{P(n)} < 2 + 2/y.
\end{align*}
Thus for all $n_i, n_j$ we have
\begin{align} \label{h/h}
    \frac{h(n_i)}{h(n_j)} < \frac{2 + 2/y}{2} = 1+ 1/y.
\end{align}

Suppose $n_i/d_i=n_j/d_j$ for some $i\neq j$. Since $n_i\neq n_j$ we have $d_i\neq d_j$. Then by multiplicativity,
\begin{align*}
    \frac{h(n_i)}{h(d_i)}=h\Big(\frac{n_i}{d_i}\Big)=h\Big(\frac{n_j}{d_j}\Big)=\frac{h(n_j)}{h(d_j)}.
\end{align*}
Since $d_i$ and $d_j$ are distinct square-free numbers, an elementary argument shows $h(d_i)\neq h(d_j)$. Without loss assume $h(d_i) > h(d_j)$. Then
$$1 < \frac{h(d_i)}{h(d_j)} = \frac{\sigma(d_i)d_j}{\sigma(d_j)d_i}$$
so that $\sigma(d_i)d_j \ge \sigma(d_j)d_i + 1$. Since $d_i$ divides the pnd $n_i$, $d_i$ is deficient so $\sigma(d_i)<2d_i$. And since $d_i\le \frac{1}{\sqrt2}y^{1/2}$ we deduce
\begin{align*}
    \frac{h(n_i)}{h(n_j)} = \frac{h(d_i)}{h(d_j)} = \frac{\sigma(d_i)d_j}{\sigma(d_j)d_i} &\ge 1 + \frac{1}{\sigma(d_j)d_i}\\
    &> 1 + \frac{1}{2d_id_j} > 1+ 1/y.
\end{align*}
 contradicting \eqref{h/h}. Hence each $n_i/d_i$ must be distinct.
\end{proof}

Combining \eqref{eq:sqfull} and Lemma \ref{lmErdos} gives our desired bound on $M(x,y)$. As before, let $\lambda = \zeta(3/2)/2\zeta(3)$.
\begin{theorem}\label{thm:M(x)}
Assume $x>y>8$. Let $b=b(x,y)$ be defined by
\begin{align}\label{dagger}
y^{-b} = 2-2\Big(1+\sqrt{2/y}\Big)^{-2\log x/\log(y/2)}.
\end{align}
Then so long as $0 < b <\tfrac{1}{2}$, we have the upper bound
\begin{align}\label{eqM}
     M(x,y) = \sideset{}{'}\sum_{\substack{n\le x\\ P(n)> y}}1 & \le (\lambda+1)xy^{-b/3} + 2xy^{-4b/9}.
\end{align}
\end{theorem}

\begin{proof}
The definition of $b$ is constructed to satisfy \eqref{eq:dagger}.
By \eqref{eq:sqfull} and Lemma \ref{lmErdos} we have
\begin{align*}
M(x,y) & = \sideset{}{'}\sum_{\substack{n\le x\\P(n)>y\\s(n)>y^a}}1 + \sideset{}{'}\sum_{\substack{n\le x\\P(n)>y\\s(n)\le y^a}}1\\
& \le \lambda xy^{-a/2}+ 2 xy^{-2a/3} + xy^{-b/3}.
\end{align*}
The result then follows from $a=2b/3$.
\end{proof}

The utility of Theorem \ref{thm:M(x)} comes to us as the following Corollary.

\begin{corollary}\label{cor:M(x)}
With $b$ as in Theorem \ref{thm:M(x)}, for $y>8$ we have that
\begin{align*}
\sideset{}{'}\sum_{\substack{x_1\le n\le x_2\\P(n)> y}}\frac{1}{n} \le (1+\log(x_2/x_1))[(\lambda+1)y^{-b/3} + 2y^{-4b/9}].
\end{align*}
\end{corollary}
\begin{proof}
Since $C = (\lambda+1)y^{-b/3} + 2 y^{-4b/9}$ is constant with respect to $x$, by partial summation and Theorem \ref{thm:M(x)},
\begin{align*}
\sideset{}{'}\sum_{\substack{x_1\le n\le x_2\\ P(n)> y}}\frac{1}{n} &= \frac{M(x_2,y)}{x_2} -  \frac{M(x_1,y)}{x_1} + \int^{x_2}_{x_1} M(x,y)\frac{dx}{x^2}\\
& \le C + C \int^{x_2}_{x_1} \frac{dx}{x} = (1+\log(x_2/x_1))C.
\end{align*}
\end{proof}

\section{Bounding the tail}

Recall the contribution of pnds less than $10^{14}$ was computed directly. On the other end, we may bound the tail of the reciprocal sum of pnds greater than $e^{5000}$. In this range, we bifurcate based on the relative size of $\omega(n)$ compared to $4\log_2(n)$. Note that if $\omega(n) \le 4\log_2(n)$, then there exists a prime power $q^a\mid n$ such that 
$$q^a \ge n^{1/4\log_2(n)}=\exp\Big({\frac{\log n}{4\log_2 n}}\Big)=:y(n).$$
If $a=1$, then $P(n)$ is large, and if $a\ge2$ then $s(n)$ is large. Thus it suffices to consider the following cases:
\begin{enumerate}
\item[(i)] $\omega(n) > 4\log_2(n)$,
\item[(ii)] $P(n) \ge y(n)$,
\item[(iii)] $s(n) \ge y(n)$.
\end{enumerate}

In case (i), by Proposition 3.2 in \cite{NP},
\begin{align}
\sum_{\substack{n>e^{5000}\\ \omega(n) >4\log_2 n}}\frac1{n} &\le \frac1{24}\sum_{k\ge 5000} \frac{(k+5)^4}{k^{4\log 4}} \le \frac{1}{24}\int_{5000}^\infty \frac{(t+5)^4}{t^{4\log 4}}\;dt \le  7.37\cdot10^{-4}. \label{eq:i}
\end{align}

In case (ii), let $y=y_k = y(e^k) = \exp\big(\frac{k}{4\log k}\big)$ and define $b=b_k=b(e^{k+1},y_k)$ from Theorem \ref{thm:M(x)}. A calculation shows that $b_k\in \big(0,\tfrac{1}{2}\big)$ for $k\ge191$. Then by Corollary \ref{cor:M(x)}
\begin{align*}
\sum_{\substack{n>e^{5000}\\ P(n) > y(n)}}\frac1{n} & \le \sum_{k\ge 5000}\sum_{\substack{e^k<n\le e^{k+1}\\ P(n)> y_k}}\frac{1}{n} \le 2\sum_{k\ge 5000}[(\lambda+1)y^{-b/3} + 2y^{-4b/9}].
\end{align*}
We may compute this sum directly up to, say, $10^4$, which contributes at most $6.8\cdot10^{-8}$. We bound the remaining series by the integral,
\begin{align*}
\sum_{k\ge 10^4}(\lambda+1)y^{-b/3} + 2y^{-4b/9} \le 2.1\int_{10^4}^\infty\exp\Big(\frac{-bt}{12\log t}\Big)\;dt.
\end{align*}
Note that the definition of $b=b_k$ in $\eqref{dagger}$ using $x=e^{k+1}$ and $y=y_k$ ensures $b\ge 0.46$ for $k\ge 10^4$. Then since $b\ge 0.46$ and $k\ge 10^4$, we have $.6\log k \ge \log(12/b) + \log_2 k$ which implies
\begin{align}\label{eq:bk}
%.6\log k &\ge \log(12/b) + \log_2 k\\
%k^{.6} &\ge (12/b)\log k\\
%k^{-.4} &\le \frac{1}{(12/b)\log k}\\
-t^{0.4} &\ge \frac{-t}{(12/b)\log t}.
\end{align}
Hence we have
\begin{align}
\sum_{\substack{n>e^{5000}\\ P(n) > y(n)}}\frac1{n} & \le 2(6.8\cdot10^{-8}) + 4.2\int_{10^4}^\infty\exp( -t^{0.4})\;dt \nonumber\\
& \le 2(6.8\cdot10^{-8}) + 10^{-14} \le 1.4\cdot10^{-7}.\label{eq:ii}
\end{align}

In case (iii), by partial summation and Lemma \ref{lmsqfull},
\begin{align}
\sum_{\substack{n>e^{5000}\\ s(n) > y(n)}}\frac1{n} & \le \sum_{k \ge 5000}\sum_{\substack{e^{k}\le n \le e^{k+1}\\ s(n) > y_k}}\frac1{n} \nonumber\\
& \le \sum_{k \ge 5000}\Big(\frac{B(e^{k+1},y_k)}{e^{k+1}} - \frac{B(e^{k},y_k)}{e^{k}} + \int_{e^k}^{e^{k+1}} B(t,y_k)\frac{dt}{t^2}\Big) \nonumber\\
& \le \sum_{k \ge 5000}(\lambda y_k^{-1/2} + 2y_k^{-2/3})\Big(1+ \int_{e^k}^{e^{k+1}} \frac{dt}{t}\Big) \le 2\sum_{k \ge 5000}(\lambda y_k^{-1/2} + 2y_k^{-2/3})  \nonumber\\
& \le 2.1\int_{5000}^\infty \exp\Big(\frac{-t}{8\log t}\Big)\;dt \le 2.1\int_{5000}^\infty \exp(-3 t^{.4})\;dt \le 6\cdot10^{-17} \label{eq:iii}
\end{align}
by a bound analogous to \eqref{eq:bk}. Hence, combining cases (i)-(iii) in equations \eqref{eq:i}-\eqref{eq:iii}, the tail is bounded by
\begin{align}
\sideset{}{'}\sum_{n>e^{5000}}\frac1{n} & \le 7.4\cdot10^{-4}.\label{eq:tail}
\end{align}

We remark that the above approach used to bound the tail from $e^{5000}$ is inadequate to use all the way from $10^{14}$. Indeed, the non-smooth contribution alone from $10^{14}$ onward is
\begin{align*}
\sideset{}{'}\sum_{\substack{n\ge 10^{14}\\ P(n)>y(n)}}\frac{1}{n}\le \sum_{k\ge 14\log10}(\lambda+1)y^{-b/3} + 2y^{-4b/9} \le 2300,
\end{align*}
which gives a very large bound (though still on par with the initial literature on the amicable numbers reciprocal sum). A bound of such magnitude reflects that it is insufficient to simply use pnd distribution estimates from \cite{erdos2} (in explicit form). Thus more refined approach is needed in order to obtain good bounds on the reciprocal sum.

We are left to deal with the contribution of pnds lying in the intermediate range $[10^{14}, e^{5000}]$. We further split up the range at $e^{700}$, and deal with the upper subrange $[e^{700}, e^{5000}]$ in the following section.

\subsection{Intermediate Range}

In the range $[e^{700}, e^{5000}]$, we implement our bounds in Corollary \ref{cor:M(x)} with greater care paid to our choice of smoothness cutoff $y$. We first estimate the reciprocal sum of smooth numbers.

\begin{lemma}\label{lma:rank}
Let $x>y\ge2$, $u=\log x/\log y$. For $u\ge10$ and $s \in [0,.041]$, we have
\begin{align*}
\sum_{\substack{n>x\\ P(n)\le y}}\frac1{n} & \le x^{-s}\exp\Big((1+\epsilon)(\mathrm{Li}(y^s) - \mathrm{Li}(2^s) + 2^s/\log2) + 0.35\Big)
\end{align*}
where $\epsilon = 2.3\times10^{-8}$.
\end{lemma}
\begin{proof}
We have
\begin{align*}
    \sum_{\substack{n>x\\ P(n)\le y}}\frac1{n} & \le x^{-s}\sum_{\substack{n>x\\ P(n)\le y}}n^{s-1} = x^{-s}\prod_{p\le y}\Big(1 - p^{s-1}\Big)^{-1}\\
    & = x^{-s}\exp\Big(-\sum_{p\le y}\log\Big(1+\frac{1}{p^{1-s}-1} \Big) \le  x^{-s}\exp\Big(\sum_{p\le y}p^{s-1} + 0.35\Big)
\end{align*}
since one verifies that for $s \le .041$,
\begin{align*}
%\sum_p \log\Big(1+\frac{1}{p^{1-s}-1} & \Big) - \frac{1}{p^{1-s}} = \bigg(\sum_{p \le 10^5} + \sum_{p > 10^5}\bigg)\log\Big(1+\frac{1}{p^{1-s}-1}\Big) - \frac{1}{p^{1-s}}\\
%& \le 0.3769 + \sum_{p\ge10^5} \Big(\frac{1}{p^{1-s}-1} - \frac{1}{p^{1-s}}\Big) \le 0.3769 + \sum_{n\ge10^5} n^{-1.8}\\
%& \le 0.3769 + \zeta(1.8) - 1.8821\le .35.
\sum_p \log(1 - p^{s-1}) - \frac{1}{p^{1-s}} \le 0.35.
\end{align*}

Let $f(t)=t^{s-1}/\log t$ and $\epsilon = 2.3\times10^{-8}$. From \cite{buthe}, we have that
\begin{align}
    \vartheta(x) = \sum_{p\le x}\log p < (1+\epsilon)x \quad \textrm{ for }x>0,
\end{align}
The result then follows from partial summation and integration by parts,
\begin{align*}
    \sum_{p\le y}p^{s-1} &= \sum_{p\le y}f(p)\log p = \vartheta(y)f(y) - \int_2^y \vartheta(t)f'(t)\;dt\\
    & \le (1+\epsilon)yf(y) - (1+\epsilon)\int_2^y tf'(t)\;dt = (1+\epsilon)2f(2) + (1+\epsilon)\int_2^y f(t)\;dt\\
    & = (1+\epsilon)(\mathrm{Li}(y^s) - \mathrm{Li}(2^s) + 2^s/\log2)
\end{align*}
using $f'(t)<0$ for $t\ge2$. This argument is adapted from \cite{NP}.
\end{proof}

We apply this bound to with $x=e^k$ for each $k\in [700,5000]$. Given $k$, we find a reasonable value of $u=u_k$, which determines $y = y_k = k/u$. Note $y_k$ here differs from the previous section. Thus by Corollary \ref{cor:M(x)} and Lemma \ref{lma:rank},
\begin{align*}
\sideset{}{'}\sum_{e^{700}\le n \le e^{5000}}\frac1{n} & \le \sum_{700\le k \le 5000}\Big(\sideset{}{'}\sum_{\substack{e^{k}\le n\le e^{k+1}\\ P(n)> y}}\frac1{n} + \sum_{\substack{n > e^k\\ P(n)\le y}}\frac1{n}\Big)\\
& \le \sum_{700\le k \le 5000}\Big(2[(\lambda+1)y^{-b/3} + 2 y^{-4b/9}] \ + \\
& \qquad\qquad\qquad \exp\Big((1+\epsilon)(\mathrm{Li}(y^s) - \mathrm{Li}(2^s) + 2^s/\log2) + .38 - sk\Big)\Big).
\end{align*}
For each $k$ we choose an optimal value of $u$, which we found to be $u = .006k+4$. We also use $s=s_k=\log(e^\gamma u\log u)/\log y$, for which one verifies $s_k\le .041$ at each $k$. With such choices of $u$ and $s$, we compute
\begin{align}
\sideset{}{'}\sum_{e^{700}\le n \le e^{5000}}\frac1{n} \le 0.001260 + 0.002237
=  0.00350.\label{eq:intermediate}
\end{align}

The remaining range $10^{14}\le n\le e^{700}$ is too small for the pnd methods to be effective. Indeed, the above method starting from $10^{14}$ gives a bound of $330$, which is orders of magnitude larger. In the next section we take a different approach, inspired by abundant density estimates.

\section{Abundant density estimates}
The density of the abundant numbers, denoted $\Delta$, is tightly bounded by
\begin{align}\label{eq:delta}
0.2476171 < \Delta < 0.2476475,
\end{align}
as proved in \cite{mits1}. Since every abundant number has a unique smallest pnd divisor, one may express the density of abundant numbers as a series over pnds, $\Delta=\sum'_{n}a_n$. Here $a_n$ is the density of all abundant numbers with smallest pnd divisor $n$. The density $a_n$ may be expressed in full via inclusion-exclusion, see \cite{mits2}, however to calculate $a_n$ this way quickly becomes unreasonable.

Towards an alternate series representation of the abundant density $\Delta$, define the {\it significance} of a number $n$ to be $\textrm{sig}(n):=\max\{\sigma(p^e):{p^e\| n}\}$. The ordering on the natural numbers by significance was first introduced in \cite{mits1}. Then letting $b_n$ be the density of abundant numbers with smallest (by significance) pnd divisor $n$, we have the series representation $\Delta=\sum'_{n}b_n$. We note that for any $x>0$, the partial sum $\sum_{n\le x}' a_n$ dominates $\sum_{n\le x}' b_n$ since there may be some pnd $n_1>x$ with significance less than that of another pnd $n_2\le x$, in which case the density of multiples of $\textrm{lcm}(n_1,n_2)$ is counted in the $a_n$ sum but not in the $b_n$ sum.\footnote{ Conversely, $\sum_{\textrm{sig}(n)\le x}' b_n$ dominates $\sum_{\textrm{sig}(n)\le x}' a_n$.}

Remarkably, when ordered by significance, the density $b_n$ factors as a product over a certain set of primes depending on $n$, see \cite{mits2}, thus enabling rapid computation. In particular, Kobayashi computed up to $4\cdot10^{10}$, giving the following.
\begin{theorem}[Theorem 3 in \cite{mits2}]
We have $\sum_{n\le 4\cdot10^{10}}' b_n = 0.24760444\ldots$. Thus the density of abundant numbers with a pnd divisor below $4\cdot10^{10}$ is at least
\begin{align}\label{eq:41010}
\sideset{}{'}\sum_{n\le 4\cdot10^{10}} a_n \ge 0.24760444.
\end{align}
\end{theorem}
Note that $0.24760444$ is a lower bound for $\Delta$, though the better bound $0.2476171$ was obtained in \cite{mits1} by a different method. Before proceeding, we prove a lemma.
\begin{lemma}\label{lma:P(n)}
The largest prime factor of a pnd $n$ is at most $\sqrt{2n}$.
\end{lemma}
\begin{proof}
Let $P(n)=p$ and suppose $p^e\| n$. If $e\ge2$, then $p^e \le n$ easily gives $p\le \sqrt{2n}$. If $e=1$, then
\begin{align*}
2n \le\sigma(n)  = \sigma(p)\sigma(n/p) & \le (p+1)(2n/p-1)%\\
%2n& \le 2n + \sigma(p^{e-1})2n/p^e - p^e-\sigma(p^{e-1})\\
%p + 1 &\le 2n/p,
\end{align*}
so $p + 1 \le 2n/p$, and thus $p\le \sqrt{2n}$.
\end{proof}

Though $b_n$ is easier to compute directly, with $a_n$ we have the useful inequality
\begin{align}\label{eq:pnddensity}
    a_n \ge \frac1{n}\prod_{p\le \sqrt{2n}}(1-1/p)\qquad\textrm{for a pnd }n.
\end{align}
Indeed, consider any number $m>1$ whose prime factors are all greater than $\sqrt{2n}$. The density of such numbers $m$ is $\prod_{p\le \sqrt{2n}}(1-1/p)$. By Lemma \ref{lma:P(n)}, $m$ is coprime to all pnds up to $n$ and so $mn$ is an abundant number with smallest pnd divisor $n$. Thus all such numbers $mn$ contribute to the density $a_n$, and so \eqref{eq:pnddensity} follows.

Now we recall Theorem 8 in \cite{RS1}, which states
\begin{align*}
\prod_{p \le x}\big(1-1/p\big) \ge \frac{e^{-\gamma}}{\log x + 1/2\log^3 x}\qquad\textrm{for }x\ge 286,
\end{align*}
and so by \eqref{eq:pnddensity} we may extract a portion of the pnd reciprocal sum,
\begin{align}
\sideset{}{'}\sum_{4\cdot10^{10}< n \le x} a_n & \ge \sideset{}{'}\sum_{4\cdot10^{10}< n \le x}\frac1{n}\prod_{p\le \sqrt{2n}}(1-1/p) \ge \sideset{}{'}\sum_{4\cdot10^{10}< n \le x} \frac{1}{n}\frac{e^{-\gamma}}{\log(2n)/2 + 4/\log^3(2n)}\nonumber\\
 & \ge \frac{e^{-\gamma}}{\log(2x)/2 + 4/\log^3(2x)}\sideset{}{'}\sum_{4\cdot10^{10} < n \le x} \frac{1}{n} \qquad\textrm{for }x\ge 286^2/2.\label{eq:dense2pnd}
\end{align}
Combining \eqref{eq:delta}, \eqref{eq:41010}, \eqref{eq:dense2pnd} yields
\begin{align*}
0.00004306 & = 0.2476475 - 0.24760444  \ge \Delta - \sideset{}{'}\sum_{n \le 4\cdot10^{10}} a_n \ge \sideset{}{'}\sum_{4\cdot10^{10}< n \le x} a_n\\
& \ge \frac{e^{-\gamma}}{\log(2x)/2 + 4/\log^3(2x)}\sideset{}{'}\sum_{4\cdot10^{10} < n \le x} \frac{1}{n},
\end{align*}
so that
\begin{align}
    \sideset{}{'}\sum_{4\cdot10^{10} < n \le x} \frac{1}{n} & \le 0.00004306\cdot e^{\gamma}(\log(2x)/2 + 4/\log^3(2x))\nonumber\\
    & \le 3.835\cdot10^{-5} \log(2x). \label{eq:abundant}
\end{align}

Finally, applying \eqref{eq:abundant} with $x=e^{700}$ leads to our desired bound.\newline
%\begin{theorem*}[Theorem \ref{thm:final}]
{\bf Theorem \ref{thm:final}.}
{\it The Erd\H{o}s constant $\sum_n'1/n$ lies in the interval}
\begin{align*}
    (0.34842 \ , \ 0.37937).
\end{align*}
%\end{theorem*}
\begin{proof}
The lower bound was computed directly up to $10^{14}$, as in \eqref{Silva}. For the upper bound, by computing up to $4\cdot10^{10}$ and combining \eqref{eq:abundant} with the results in \eqref{eq:tail} and \eqref{eq:intermediate},
\begin{align*}
\sideset{}{'}\sum_{n}\frac{1}{n} & =  \sideset{}{'}\sum_{n \le 4\cdot10^{10}}\frac{1}{n} + \sideset{}{'}\sum_{4\cdot10^{10}< n \le e^{700}}\frac{1}{n} + \sideset{}{'}\sum_{e^{700}< n \le e^{5000}}\frac{1}{n} + \sideset{}{'}\sum_{n >e^{5000}}\frac{1}{n} \\
&\le 0.348255 + 0.026872 + 0.00350 + 0.00074 \ \le \ 0.37937.%\hskip4.5in\square
\end{align*}
\end{proof}

\section*{Acknowledgments}
I would like to thank Carl Pomerance for his nondeficient supply of ideas, encouragement, and attention to detail. I thank Mits Kobayashi for sharing his computations, as well as his helpful conversations early on. I also thank Tom\'as Silva, who recently extended the reciprocal sum calculation out to $10^{14}$. The author was partially supported by a Byrne Scholarship at Dartmouth. 

\bibliographystyle{amsplain}

\end{document}